\documentclass[a4paper,11pt]{article} 
\usepackage{amsmath,amssymb,enumerate,color}
\usepackage{amsthm,color}
\usepackage[colorlinks=true,linkcolor=red,citecolor=blue]{hyperref}
\usepackage{txfonts}
\usepackage{bigints}
\usepackage{comment}
\usepackage{graphicx}
\usepackage{enumitem}

\usepackage{bm}

\newcommand{\N}{\mathbb{N}}

\newtheorem{theorem}{Theorem}[section]
\newtheorem{lemma}[theorem]{Lemma}
\newtheorem{proposition}[theorem]{Proposition}

\theoremstyle{remark}
\newtheorem{remark}{Remark}[section]
\theoremstyle{definition}

\newtheorem{definition}{Definition}[section]

\numberwithin{equation}{section}

\makeatletter
\def\@cite#1#2{[{{\bfseries #1}\if@tempswa , #2\fi}]}
\makeatother                  
\begin{document}
\begin{center}
\Large{{\bf
Global existence of solutions for nonlinear damped wave equations in an exterior domain with nonlinearities of derivative type}}
\end{center}

\vspace{5pt}

\begin{center}
Tuan Anh Dao
\footnote{
Faculty of Mathematics and Informatics, Hanoi University of Science and Technology, No.1 Dai Co Viet road, Hanoi, Vietnam.
E-mail:\ {\tt anh.daotuan@hust.edu.vn} (\textit{Corresponding author})},
Dinh Van Duong
\footnote{
Faculty of Mathematics and Informatics, Hanoi University of Science and Technology, No.1 Dai Co Viet road, Hanoi, Vietnam.
E-mail:\ {\tt van.duongdinh@hust.edu.vn}},
Masahiro Ikeda
\footnote{
Center for Advanced Intelligence Project, RIKEN 1-4-1, Nihonbashi, Chuo-ku, Tokyo 103-0027, Japan/Graduate School of Information Science and Technology, The University of Osaka, 1-5 Yamadaoka, 565-0871, Suita, Osaka, Japan.
E-mail:\ {\tt masahiro.ikeda@riken.jp}}\ 

\end{center}

\newenvironment{summary}{\vspace{.5\baselineskip}\begin{list}{}{%
     \setlength{\baselineskip}{0.85\baselineskip}
     \setlength{\topsep}{0pt}
     \setlength{\leftmargin}{12mm}
     \setlength{\rightmargin}{12mm}
     \setlength{\listparindent}{0mm}
     \setlength{\itemindent}{\listparindent}
     \setlength{\parsep}{0pt}
     \item\relax}}{\end{list}\vspace{.5\baselineskip}}

\begin{summary}
{\footnotesize {\bf Abstract.}
In this paper, we are interested in considering the semi-linear damped wave equation with the power nonlinearity of derivative type in $2$D exterior domains to indicate the global (in time) existence of small data solutions, which has never appeared in previous studies. Our proof is based on constructing a judicious time-weighted function space and demonstrating nonlinear estimates compatible with fractional powers of the Dirichlet Laplacian linked to the semigroup decay structure. Moreover, the large time behavior of the time derivative of the obtained global solutions and its sharp estimate are also discussed in this paper.}
\end{summary}

\tableofcontents

\section{Introduction}

Let $\Omega\subset\mathbb{R}^2$ be a $\mathcal{C}^{1,1}$ exterior domain, namely, $\Omega=\mathbb{R}^2\setminus \overline{\mathcal{O}}$, where $\mathcal{O}\subset\mathbb{R}^2$ is a bounded open set and $\partial \mathcal{O}$ is of class $\mathcal{C}^{1,1}$. Let $\mathcal{A}:=-\Delta_{D,\Omega}$ denote the Dirichlet Laplacian on $\Omega$, whose domain is given by
\[
D(\mathcal{A})
:=
\{
u\in H^1_0(\Omega)
\;:\;
\Delta u\in L^2(\Omega)
\}.
\]
For $u\in D(\mathcal{A})$, it holds that $\mathcal{A}u=-\Delta u$. In this paper, we would like to study the global (in time) existence of solutions to the following initial-boundary value problem for the semi-linear damped wave equation with the power nonlinearity of derivative type in $2$D exterior domains:
\begin{equation}\label{eq:main}
\begin{cases}
u_{tt}+u_t+ \mathcal{A}u= \mathcal{N}[u_t], & (t,x)\in(0,\infty)\times\Omega,\\
u(t,x)=0, &(t,x)\in(0,\infty)\times \partial\Omega, \\
u(0,x)=u_0(x),\quad u_t(0,x)=u_1(x), & x\in \Omega.
\end{cases}
\end{equation}
Here for $p>1$, the nonlinear function $\mathcal{N}:\mathbb{C}\rightarrow\mathbb{C}$ denotes the $p$-th order power nonlinearity such as $\mathcal{N}[z]=\pm |z|^p$ or $\pm |z|^{p-1}z$ for typical examples. The functions
$u_0:\Omega\to\mathbb C$ and $u_1:\Omega\to\mathbb C$
represent the initial displacement and initial velocity, respectively. We are interested in the global existence of solutions generated by sufficiently small initial data in suitable energy-type spaces. To understand our motivation in this work, let us make a brief overview of previous studies involving damped wave equations as follows:

The semi-linear damped wave equation has been extensively studied as a model describing the interaction between dissipation phenomena and nonlinear propagation. One of the fundamental observations is that solutions asymptotically exhibit diffusion-like behavior, and therefore their global dynamics are closely related to those of semi-linear heat equations. For the Cauchy problem in the whole space, a lot of results for global/local (in time) existence, asymptotic behavior, blow-up, and even lifespan estimates are so far well-known. In particular, Matsumura in \cite{Matsumura1976} established decay estimates for the corresponding linear damped wave equation, which became one of fundamental tools in terms of the study of nonlinear problems. Later, both the global (in time) existence of solutions and the critical exponent phenomena were widely investigated by many authors (see, for examples, \cite{Nishihara2003,IkehataTanizawa2005,Todorova2001}). According to the works \cite{Iked-Ogaw-2016,Lai-Zhou-2019,Fuji-Iked-Waka-2019,Iked-Waka-2020}, the authors demonstrated the sharp lifespan estimates for all spatial dimensions. More recently, Ikeda–Taniguchi–Wakasugi in \cite{IkedaTaniguchiWakasugi2023} considered an abstract framework for semi-linear damped wave equations on measure spaces and derived linear decay estimates of Matsumura-type under general assumptions on the heat semigroup. As an application, they proved the global (in time) existence of small data solutions to semi-linear damped wave equations with the usual power nonlinearity of the form $|u|^p$, with $p>1$, including the case of exterior domains accompanied by the Dirichlet Laplacian. In comparison with Cauchy problems, we can say that exterior domain problems show additional difficulties caused by the lack of translation invariance and the interaction with the boundary conditions. On the other hand, Fourier analysis is no longer directly applicable so that one needs alternative approaches based on semigroup theory and spectral analysis. Speaking about such kind of results for semi-linear damped wave equations in exterior domains, we want to refer the readers to several papers, for instance, \cite{Ikehata2005,OgawaTakeda2009,Fino2017,IkedaSobajima2019-1,IkedaJleliSamet2020,DaoIkeda2023} and references therein. More precisely, under the Dirichlet boundary
condition Ikehata in \cite{Ikehata2005} derived a global existence result for compactly supported initial data having a small energy in $2$D, provided that the condition $p>2$ holds. Afterward, Ogawa-Takeda in \cite{OgawaTakeda2009} and Fino-Ibrahim-Wehbe in \cite{Fino2017} applied the Kaplan-Fujita method to indicate non-existence of non-negative global solutions for any dimension when $p\in (1,1+2/n)$ and $p=1+2/n$, respectively. Then, Ikeda-Sobajima in \cite{IkedaSobajima2019-1} have succeeded in catching sharp upper lifespan estimates for $2$D exterior problems and clarified the exceptional double-exponential behavior at the critical exponent $p=2$. This phenomenon reflects the recurrence property of Brownian motion in $2$D and differs significantly from higher-dimensional cases. Moreover, Ikeda-Jleli-Samet in \cite{IkedaJleliSamet2020} explored both existence and nonexistence phenomena for semi-linear exterior problems under several different boundary conditions and showed how the boundary geometry affects the critical behavior. Very recently, Ikeda–Sobajima–Taniguchi–Wakasugi in \cite{IkedaSobajimaTaniguchiWakasugi2024} not only have refined the lifespan analysis for semi-linear damped wave equations in $2$D exterior domains but also have obtained sharp lower estimates together with new weighted inequalities adapted to the logarithmic decay structure.

Despite these developments, the fact is that almost previous studies only focus on nonlinear terms depending on the solution itself, typically of the form $|u|^p$. In contrast, much less is known for derivative-type nonlinearities such as $|u_t|^p$ even some existing literature also take into consideration of Cauchy problems (see more \cite{Matsumura1976,DuongDao}). In particular, Matsumura in \cite{Matsumura1976} established the global (in time) existence of solutions for the entire range $p > 1$ ($n = 1$) and $p \ge 2$ ($n \ge 2$), which are improved in the larger range for all $p > 1$ ($n = 1,2$) and for $p > 3/2$ ($n = 3$) by Duong-Dao \cite{DuongDao}. This improvement is relied on choosing a suitable weighted solution space and effectively employ ingredients from Harmonic Analysis linked to the Banach fixed-point theorem. From these observations, we can say that derivative-type nonlinearity directly interacts with the dissipative structure, and furthermore, gives some difficulties by requiring higher regularity estimates together with nonlinear estimates for fractional powers of the Dirichlet Laplacian in the situation of exterior domains. Inspired strongly by \cite{DuongDao}, the purpose of the present paper is to demonstrate a result for the global (in time) existence of small data solutions to \eqref{eq:main}. Our approach is based on a combination between decay estimates for the damped linear semigroup and fractional nonlinear estimates in Sobolev spaces associated with the Dirichlet Laplacian. More precisely, we want to develop the fractional chain rule on the whole space to its new version in the exterior domain setting (see later, Proposition \ref{prop:chain}), which plays a key point in our proof. To the best of our knowledge, this work provides the first global existence result for semi-linear damped wave equations with the power nonlinearity of derivative type in exterior domains under fractional regularity assumptions. Among other things, the point we should underline in the this paper is that both the asymptotic profile and the optimal estimate for the time derivative of the derived global solutions to \eqref{eq:main} are also investigated well. 

Our main result is the following.

\begin{theorem} \label{Main.Theorem}
Let \(\Omega\subset\mathbb R^2\) be a $C^{1,1}$ exterior domain. Assume that $p>1$ and the nonlinear function $\mathcal{N}$ satisfies the same assumptions as in Section \ref{Pre.Sec}. Moreover, suppose that the initial data belong to the class
\[
u_0\in L^{\gamma(p)}(\Omega)\cap H^{\sigma+1}(\mathcal{A}),
\quad
u_1\in L^{\gamma(p)}(\Omega)\cap H^\sigma(\mathcal{A}),
\]
where $\sigma\in (1,\min\{2,p\})$ and
\[
\gamma(p):=
\begin{cases}
2/p &\text{ if } 1<p<2,\\
1 &\text{ if } p\ge2,
\end{cases}
\]
carrying the norm $\|(u_0, u_1)\|_{\rm Data}:= \|u_0\|_{L^{\gamma(p)}(\Omega)}+\|u_0\|_{H^{\sigma+1}(\mathcal{A})}+\|u_1\|_{L^{\gamma(p)}(\Omega)}+\|u_1\|_{H^{\sigma}(\mathcal{A})}$. Then, there exists a constant $\varepsilon_0 > 0$ such that for any small data $\|(u_0, u_1\|< \varepsilon_0$, we have a uniquely determined global (in time) small data mild solution
\[
u\in \mathcal{C}^1\Big([0,\infty);D(\mathcal{A}^{\sigma/2})\Big).
\]
to \eqref{eq:main} and the following estimates hold:
        \begin{equation}\label{Decay.Es}
        \begin{split}
            \|u_t(t,\cdot)\|_{L^2(\Omega)} &\lesssim (1+t)^{-\gamma}\|(u_0, u_1)\|_{\rm Data}, \\
            \|u_t(t,\cdot)\|_{\dot{H}^\sigma(\mathcal{A})} &\lesssim (1+t)^{-\eta}\|(u_0, u_1)\|_{\rm Data},
        \end{split}
        \end{equation}
where two positive constants $\gamma,\,\eta$ are defined as in Section \ref{Proof.Sec}. Moreover, the global obtained solutions to \eqref{eq:main} satisfy the following refined decay estimates:
\begin{equation}
\label{eq:ut-diffusion-limit}
\lim_{t\to\infty}
t^\gamma
\left\|
u_t(t)+\mathcal{A}e^{-t\mathcal{A}}(u_0+u_1+M_\infty)
\right\|_{L^2(\Omega)}
=
0,
\end{equation}
where
$$M_\infty
:=
\int_0^\infty \mathcal{N}[u_t](s)\,ds. $$
\end{theorem}

\begin{remark}
From the statement of Theorem \ref{Main.Theorem}, we may claim that the the small data solution always exist globally for any $p>1$ in terms of studying the initial-boundary value problem for a semi-linear damped wave equation in $2$D exterior domains. This phenomenon been also appeared in \cite{DuongDao} concerning the corresponding Cauchy problem of \eqref{eq:main}.
\end{remark}

\begin{remark}
One recognizes that there does not exist the critical exponent $p_{\rm crit}>1$ for \eqref{eq:main} in which the power nonlinearity of derivative type $|u_t|^p$ is of interest. In a comparison with the usual power nonlinearity $|u|^p$, we recall the critical exponent $p_{\rm crit}=2$ when $n=2$. This observation tells us that the power nonlinearity of derivative type influences more strongly than the usual power nonlinearity with respect to an admissible range of power exponents for the global existence.
\end{remark}

\begin{remark}
    We want to stress out that the diffusion profile of the time derivative of solutions in \eqref{eq:ut-diffusion-limit} also give the following lower bound estimate:
    $$ \|u_t(t,\cdot)\|_{L^2(\Omega)} \ge C\|u_0+u_1+M_\infty\|_{L^1(\Omega)}(1+t)^{-\gamma}, $$
    provided that $u_0+u_1+M_\infty \not\equiv 0$ together with additional assumptions of $u_0,u_1,M_\infty\in L^1(\Omega)$, by the aid of linear estimates for heat semigroup. Linking this to the upper bound estimate \eqref{Decay.Es}, it means that the sharp decay estimate for $u_t$ is given by
    $$ \|u_t(t,\cdot)\|_{L^2(\Omega)} \sim (1+t)^{-\gamma}. $$
\end{remark}

\noindent\textbf{The organization of this paper is as follows:} In Section \ref{Pre.Sec}, we present preliminary knowledge as some notations, fractional chain rule associated with the Dirichlet Laplacian and Matsumura-type estimates for solutions to the corresponding linear equation of \eqref{eq:main}. Then, we are going to prove the main result in Section \ref{Sec.Proof} including the global (in time) existence of small data solutions as well as the large time behavior property of these global solutions to \eqref{eq:main}. Finally, several open problems will be proposed in Section \ref{Sec.Final}.

\section{Preliminaries}\label{Pre.Sec}

Throughout the paper, for $1\le q\le\infty$, we denote by $L^q(\Omega)$, the usual Lebesgue space with norm
$$
\|f\|_{L^q(\Omega)}
=
\begin{cases}
    \left(\displaystyle\int_\Omega |f(x)|^q,dx \right)^{1/q} &\text{ if } 1\le q<\infty, \\
    \operatorname*{ess,sup}_{x\in\Omega} |f(x)| &\text{ if } q= \infty.
\end{cases}
$$
We define the Dirichlet Laplacian through the theory of closed quadratic forms. Namely, let us consider the sesquilinear form
\begin{equation*}
\mathfrak a[u,v]
:=
\int_{\Omega}
\nabla u(x)\cdot\overline{\nabla v(x)}\,dx,
\qquad
u,v\in D(\mathfrak a):=H_0^1(\Omega).
\end{equation*}
This form \(\mathfrak a\) is densely defined, symmetric, nonnegative, and closed on \(L^2(\Omega)\). By the representation theorem for closed nonnegative forms, there
exists a unique nonnegative self-adjoint operator  \(\mathcal{A}\) on \(L^2(\Omega)\) such that
\begin{equation*}
\mathfrak a[u,v]
=
(\mathcal{A}u,v)_{L^2(\Omega)}
\end{equation*}
for every \(u\in D(\mathcal{A})\) and \(v\in H_0^1(\Omega)\), where
\begin{equation*}
D(\mathcal{A})
=
\left\{
u\in H_0^1(\Omega):
\text{there exists } f\in L^2(\Omega)
\text{ such that }
\mathfrak a[u,v]=(f,v)_{L^2(\Omega)}
\text{ for all }v\in H_0^1(\Omega)
\right\},
\end{equation*}
and \(\mathcal{A}u:=f\). The uniqueness of \(f\) follows from the density of \(H_0^1(\Omega)\) in \(L^2(\Omega)\). This operator is the Dirichlet realization of the negative Laplacian, and we write $\mathcal{A}=-\Delta_{D,\Omega}$. In particular, in the distributional sense we have
\[
\mathcal{A}u=-\Delta u,
\qquad
u\in D(\mathcal{A}).
\]
If \(\Omega\) is of class \(\mathcal{C}^{1,1}\), the elliptic regularity yields
\begin{equation*}
D(\mathcal{A})
=
H^2(\Omega)\cap H_0^1(\Omega),
\end{equation*}
with equivalence of the graph norm of \(\mathcal{A}\) and the \(H^2(\Omega)\)-norm on \(D(\mathcal{A})\).

Since \(\mathcal{A}\) is nonnegative and self-adjoint, the spectral theorem provides a projection-valued measure \(E_\mathcal{A}\) on \([0,\infty)\) such that
\begin{equation*}
\mathcal{A}
=
\int_{[0,\infty)}
\lambda\,dE_\mathcal{A}(\lambda).
\end{equation*}
For \(u\in L^2(\Omega)\), define the finite positive Borel measure
\[
\mu_u(B)
:=
\big\|E_\mathcal{A}(B)u\big\|_{L^2(\Omega)}^2,
\qquad
B\subset[0,\infty)
\]
for every Borel set \(B\). For \(r\ge0\), the fractional power
\(\mathcal{A}^{r/2}\) is defined by
\begin{equation*}
\mathcal{A}^{r/2}u
:=
\int_{[0,\infty)}
\lambda^{r/2}\,dE_\mathcal{A}(\lambda)u
\end{equation*}
on the domain
\begin{equation*}
D(\mathcal{A}^{r/2})
:=
\left\{
u\in L^2(\Omega):
\int_{[0,\infty)}
\lambda^r\,d\mu_u(\lambda)<\infty
\right\}.
\end{equation*}
Equivalently, it holds
\[
u\in D(\mathcal{A}^{r/2})
\quad\text{ if and only if }\quad
\int_{[0,\infty)}
\lambda^r
\,d\|E_\mathcal{A}(\lambda)u\|_{L^2(\Omega)}^2
<\infty
\]
and
\begin{equation*}
\big\|\mathcal{A}^{r/2}u\big\|_{L^2(\Omega)}^2
=
\int_{[0,\infty)}
\lambda^r\,d\mu_u(\lambda)
\end{equation*}
for every \(u\in D(\mathcal{A}^{r/2})\). Thus, \(\mathcal{A}^{r/2}\) is a densely defined, nonnegative, self-adjoint, and closed operator on \(L^2(\Omega)\). For \(r\ge0\), we define
\begin{equation*}
H^r(\mathcal{A})
:=
D(\mathcal{A}^{r/2})
\end{equation*}
and equip this space with the norm
\begin{equation*}
\|u\|_{H^r(\mathcal{A})}
:=
\|(\mathcal{I}+\mathcal{A})^{r/2}u\|_{L^2(\Omega)}.
\end{equation*}
By the spectral theorem, one sees that
\begin{equation*}
\|u\|_{H^r(\mathcal{A})}^2
=
\int_{[0,\infty)}
(1+\lambda)^r\,d\mu_u(\lambda),
\end{equation*}
together with the corresponding inner product
\begin{equation*}
(u,v)_{H^r(\mathcal{A})}
:=
\Big(
(\mathcal{I}+\mathcal{A})^{r/2}u,
(\mathcal{I}+\mathcal{A})^{r/2}v
\Big)_{L^2(\Omega)}.
\end{equation*}
Moreover, \(H^r(\mathcal{A})\) is a Hilbert space with respect to this inner product. 

For \(0<r<2\), \(D(\mathcal{A}^{r/2})\) is identified with the corresponding Dirichlet Sobolev space \(H^r_D(\Omega)\) and with the appropriate trace condition. Denoting the homogeneous H\"older space $\dot{\mathcal{C}}^\alpha(\mathbb R^2)$ for $0<\alpha<1$, whose seminorm is defined by
$$
\|f\|_{\dot{\mathcal{C}}^\alpha(\mathbb R^2)}
:=
\sup_{x\neq y}
\frac{|f(x)-f(y)|}
{|x-y|^\alpha}.
$$
We identify \(\mathbb C\) with \(\mathbb R^2\) and regard the nonlinear function \(\mathcal N\) as a mapping from \(\mathbb R^2\) into either
\(\mathbb R\) or \(\mathbb R^2\). Also, we introduce the following assumptions on the nonlinearity:

\begin{itemize}

\item $\mathcal N(0)=0$.

\item There exists a constant \(C_{\mathcal N}>0\) and \(p>1\) such that for all \(z_1,z_2\in\mathbb R^2\)
\begin{equation}
\label{eq:N-Lip}
|\mathcal N(z_1)-\mathcal N(z_2)|
\le
C_{\mathcal N}
\bigl(|z_1|+|z_2|\bigr)^{p-1}
|z_1-z_2|.
\end{equation}

\item Let \(1<\sigma<\min\{2,p\}\). We assume that $
\mathcal N\in \mathcal C^1(\mathbb R^2,\mathbb R^m)$ for $m\in\{1,2\},
$
and that there exists \(C_{\mathcal N}>0\) such that
\begin{equation}
\label{eq:DN-growth}
\big\|D\mathcal N(z)\big\|_{\mathcal C(\mathbb R^2,\mathbb R^m)}
\le
C_{\mathcal N}|z|^{p-1}
\end{equation}
for every \(z\in\mathbb R^2\). Moreover, \(D\mathcal N\) is assumed to be locally H\"older continuous
of order \(\sigma-1\) fulfilling
\begin{equation*}
\big\|D\mathcal N(z_1)-D\mathcal N(z_2)\big\|_{\mathcal C(\mathbb R^2,\mathbb R^m)}
\le
C_{\mathcal N}
\bigl(|z_1|+|z_2|\bigr)^{p-\sigma}
|z_1-z_2|^{\sigma-1}
\end{equation*}
for all \(z_1,z_2\in\mathbb R^2\).

\item If \(p>2\), we assume in addition that $\mathcal N\in \mathcal C^2(\mathbb R^2,\mathbb R^m)$
and for every \(z\in\mathbb R^2\) it holds
\begin{equation}
\label{eq:D2N-growth}
\big\|D^2\mathcal N(z)\big\|_{\mathcal C(\mathbb R^2,\mathbb R^m)}
\le
C_{\mathcal N}|z|^{p-2},
\end{equation}
where $D^2\mathcal N(z)\in
\mathcal C(\mathbb R^2,\mathbb R^m)$ denotes the second Fr\'echet derivative, regarded as a bounded bilinear mapping.
\end{itemize}
For later convenience, hereafter $C$ and $C_j$, with $j\in\N$, stand for suitable positive constants which may be changed from line to line.

\subsection{Fractional chain rule}

We first recall the fractional Leibniz rule on $\mathbb R^2$. This estimate will later be transferred to the setting of exterior domains via an extension argument.

\begin{lemma}[Fractional Leibniz rule on $\mathbb R^2$, see Lemma 2.7 in \cite{IkedaInuiWakasugi2026}]
\label{lem:Leibniz}
Let $\alpha\in (0,1)$. We assume
\[
f\in L^\infty(\mathbb R^2)\cap \dot{\mathcal{C}}^\alpha(\mathbb R^2)
\quad\text{ and }\quad
g\in \dot H^\alpha(\mathbb R^2)\cap L^2(\mathbb R^2),
\]
then there exists a positive constant $C$ depending only on $s$ such that
\begin{equation*}
\|fg\|_{\dot H^\alpha(\mathbb R^2)}
\le
C
\|f\|_{L^\infty(\mathbb R^2)}
\|g\|_{\dot H^\alpha(\mathbb R^2)}
+
C
\|f\|_{\dot{\mathcal{C}}^\alpha(\mathbb R^2)}
\|g\|_{L^2(\mathbb R^2)}.
\end{equation*}
\end{lemma}

\begin{proposition}[Fractional chain rule associated with the Dirichlet Laplacian]
\label{prop:chain}
Let $0<\sigma<2$ and $p>\sigma$. Assume that the nonlinear function $\mathcal{N}$ satisfies the above assumptions. Then, the following estimates hold:
\begin{itemize}
    \item[i)] If $f\in H^\sigma(\mathcal{A})\cap L^\infty(\Omega)$, then $\mathcal{N}[f]\in H^\sigma(\mathcal{A})$ and
    \begin{equation}
    \label{eq:chain}
    \big\|\mathcal{N}[f]\big\|_{H^\sigma(\mathcal{A})}
    \le
    C
    \|f\|_{L^\infty(\Omega)}^{p-1}
    \|f\|_{H^\sigma(\mathcal{A})}.
    \end{equation}
    \item[ii)] If $f,g\in D(\mathcal A^{\sigma/2})\cap L^\infty(\Omega)$, then $\mathcal N[f]-\mathcal N[g]\in D(\mathcal A^{\sigma/2})$
and
\begin{align}
&\bigl\|
\mathcal A^{\sigma/2}
\bigl(
\mathcal N[f]-\mathcal N[g]
\bigr)
\bigr\|_{L^2(\Omega)} \notag \\
&\quad\le
C
\left(
\|f\|_{L^\infty(\Omega)}^{p-1}
+
\|g\|_{L^\infty(\Omega)}^{p-1}
\right)
\bigl\|
\mathcal A^{\sigma/2}(f-g)
\bigr\|_{L^2(\Omega)}
\label{eq:Lip-chain-first}\\
&\qquad
+
C
\left(
\|f\|_{L^\infty(\Omega)}^{p-2}
+
\|g\|_{L^\infty(\Omega)}^{p-2}
\right)
\left(
\|\mathcal A^{\sigma/2}f\|_{L^2(\Omega)}
+
\|\mathcal A^{\sigma/2}g\|_{L^2(\Omega)}
\right)
\|v-w\|_{L^\infty(\Omega)}.
\notag
\end{align}
\end{itemize}
\end{proposition}

\begin{proof}
\textit{At first, let us verify the estimate \eqref{eq:chain}}. Namely, we take an extension operator \(\mathcal{E}\) satisfying $\mathcal{E}:H^\sigma_D(\Omega)\cap L^\infty(\Omega)
\to
H^\sigma(\mathbb R^2)\cap L^\infty(\mathbb R^2)$ and
\[
\big\|\mathcal{E}[f]\big\|_{H^\sigma(\mathbb R^2)}
\le C\|f\|_{H^\sigma_D(\Omega)},
\qquad
\big\|\mathcal{E}[f]\big\|_{L^\infty(\mathbb R^2)}
\le C\|f\|_{L^\infty(\Omega)}.
\]
Since $\|f\|_{H^\sigma(\mathcal{A})}
\sim
\|f\|_{H^\sigma(\Omega)}$, it suffices to prove
$$
\big\|\mathcal{N}[f]\big\|_{H^\sigma(\mathbb R^2)}
\le
C
\|f\|_{L^\infty(\mathbb R^2)}^{p-1}
\|f\|_{H^\sigma(\mathbb R^2)}.$$
For $0<\sigma<1$, using the Gagliardo characterization one has
$$
[f]_{\dot H^\sigma(\mathbb R^2)}^2
\sim
\iint_{\mathbb R^2\times\mathbb R^2}
\frac{|f(x)-f(y)|^2}{|x-y|^{2+2\sigma}}
\,dx\,dy.
$$
By \eqref{eq:N-Lip}, we obtain
$$
\big|\mathcal{N}[f(x)]-\mathcal{N}[f(y)]\big|
\le
C
\|f\|_{L^\infty(\mathbb R^2)}^{p-1}
|f(x)-f(y)|,
$$
which gives
$$
\big[\mathcal{N}[f]\big]_{\dot H^\sigma(\mathbb R^2)}
\le
C
\|f\|_{L^\infty(\mathbb R^2)}^{p-1}
[f]_{\dot H^\sigma(\mathbb R^2)}.
$$
Since $\mathcal{N}(0)=0$, it follows from \eqref{eq:N-Lip} that
$$
\big\|\mathcal{N}[f]\big\|_{L^2(\mathbb R^2)}
\le
C
\|f\|_{L^\infty(\mathbb R^2)}^{p-1}
\|f\|_{L^2(\mathbb R^2)},
$$
thus
$$
\big\|\mathcal{N}[f]\big\|_{H^\sigma(\mathbb R^2)}
\le
C
\|f\|_{L^\infty(\mathbb R^2)}^{p-1}
\|f\|_{H^\sigma(\mathbb R^2)}.
$$
Next, for $1<\sigma<2$ we put $\alpha=\sigma-1$. Since $\mathcal{N}$ is regarded as a $\mathcal{C}^1$ map from $\mathbb R^2$ to $\mathbb R^2$, we have
$$
\nabla \mathcal{N}[f] = D\mathcal{N}[f]\nabla f.
$$
By Lemma~\ref{lem:Leibniz}, we may estimate
\begin{align*}
\big\|\nabla \mathcal{N}[f]\big\|_{\dot H^\alpha(\mathbb R^2)}
&\le
C
\big\|D\mathcal{N}[f]\big\|_{L^\infty(\mathbb R^2)}
\|\nabla f\|_{\dot H^\alpha(\mathbb R^2)}
+
C
\big\|D\mathcal{N}[f]\big\|_{\dot{\mathcal{C}}^\alpha(\mathbb R^2)}
\|\nabla f\|_{L^2(\mathbb R^2)} \\
&\le
C
\|f\|_{L^\infty(\mathbb R^2)}^{p-1}
\|f\|_{H^\sigma(\mathbb R^2)}
\end{align*}
together with
$$
\big\|\mathcal{N}[f]\big\|_{L^2(\mathbb R^2)}
\le
C
\|f\|_{L^\infty(\mathbb R^2)}^{p-1}
\|f\|_{L^2(\mathbb R^2)},
$$
to arrive at
$$
\big\|\mathcal{N}[f]\big\|_{H^\sigma(\mathbb R^2)}
\le
C
\|f\|_{L^\infty(\mathbb R^2)}^{p-1}
\|f\|_{H^\sigma(\mathbb R^2)}.
$$
In this way, turning to $\Omega$ we may conclude the estimate \eqref{eq:chain} to complete our proof. \medskip

\textit{Next, let us verify the estimate \eqref{eq:Lip-chain-first}}. Indeed, we divide the proof into two steps as follows. \medskip

\noindent\textbf{$\bullet$ Step 1:} Setting $h:=f-g$ we rewrite the nonlinear difference in the form
\begin{align*}
\mathcal{N}[f(x)]-\mathcal{N}[g(x)]
=
\int_0^1
\frac{d}{d\theta}
\mathcal{N}[g(x)+\theta h(x)]
\,d\theta
&=
\int_0^1
D\mathcal{N}[g(x)+\theta h(x)]
h(x)\,d\theta \\
&= \int_0^1 D\mathcal{N}[h_\theta(x)]\, h(x)\,d\theta
\end{align*}
with $
h_\theta:=g+\theta h=(1-\theta)g+\theta f$. Then, applying the fractional Leibniz rule we obtain
\begin{equation}
\label{eq:product-B-h}
\big\|D\mathcal{N}[h_\theta] h\big\|_{\dot H^\sigma(\Omega)}
\le
C\big\|D\mathcal{N}[h_\theta]\big\|_{L^\infty(\Omega)}
\|h\|_{\dot H^\sigma(\Omega)}
+
C\|h\|_{L^\infty(\Omega)}
\big\|D\mathcal{N}[h_\theta]\big\|_{\dot H^\sigma(\Omega)}
\end{equation}
for \(1<\sigma<2\). Let us now estimate \(D\mathcal{N}[h_\theta]\) in \(L^\infty(\Omega)\) and \(\dot H^\sigma(\Omega)\). In particular, by \eqref{eq:DN-growth} one has
\[
\big|D\mathcal{N}[h_\theta(x)]\big|
\le
C|h_\theta(x)|^{p-1}
\le
C_p
\left(
|f(x)|^{p-1}+|g(x)|^{p-1}
\right),
\]
which leads to
\begin{equation}
\label{eq:Btheta-Linfty}
\big\|D\mathcal{N}[h_\theta]\big\|_{L^\infty(\Omega)}
\le
C
\left(
\|f\|_{L^\infty(\Omega)}^{p-1}
+
\|g\|_{L^\infty(\Omega)}^{p-1}
\right).
\end{equation}
Substituting \eqref{eq:Btheta-Linfty} into the first term of
\eqref{eq:product-B-h}, we obtain
\begin{align}
&\big\|D\mathcal{N}[h_\theta]\big\|_{L^\infty(\Omega)}
\|h\|_{\dot H^\sigma(\Omega)} \le
C
\left(
\|f\|_{L^\infty(\Omega)}^{p-1}
+
\|g\|_{L^\infty(\Omega)}^{p-1}
\right)
\|f-g\|_{\dot H^\sigma(\Omega)},
\label{eq:first-product-term}
\end{align}
which gives the first term on the right-hand side of
\eqref{eq:Lip-chain-first}. On the other hand, it holds from the standard Sobolev composition estimate that
\begin{equation}
\label{eq:composition-DN}
\big\|D\mathcal{N}[h_\theta]\big\|_{\dot H^\sigma(\Omega)}
\le
C
\big\|D^2\mathcal{N}[h_\theta]\big\|_{L^\infty(\Omega)}
\|h_\theta\|_{\dot H^\sigma(\Omega)}.
\end{equation}
By \eqref{eq:D2N-growth}, we derive
\[
\big|D^2\mathcal{N}[h_\theta(x)]\big|
\le
C|h_\theta(x)|^{p-2}
\le
C_p
\left(
|f(x)|^{p-2}+|g(x)|^{p-2}
\right)
\]
since \(p\ge2\). Thus, it follows
\begin{equation*}
\big\|D^2\mathcal N[h_\theta]\big\|_{L^\infty(\Omega)}
\le
C
\left(
\|f\|_{L^\infty(\Omega)}^{p-2}
+
\|g\|_{L^\infty(\Omega)}^{p-2}
\right),
\end{equation*}
which implies from \eqref{eq:composition-DN} that
\begin{align*}
\big\|D\mathcal{N}[h_\theta]\big\|_{\dot H^\sigma(\Omega)}
&\le
C
\left(
\|f\|_{L^\infty(\Omega)}^{p-2}
+
\|g\|_{L^\infty(\Omega)}^{p-2}
\right) \left(
\|f\|_{\dot H^\sigma(\Omega)}
+
\|g\|_{\dot H^\sigma(\Omega)}
\right)
\end{align*}
thanks to the triangle inequality $\|h_\theta\|_{\dot H^\sigma(\Omega)} \le
\|f\|_{\dot H^\sigma(\Omega)}
+
\|g\|_{\dot H^\sigma(\Omega)}$. Therefore, the second term of \eqref{eq:product-B-h} is controlled by
\begin{align}
&\|h\|_{L^\infty(\Omega)}
\big\|D\mathcal{N}[h_\theta]\big\|_{\dot H^\sigma(\Omega)}
\le
C
\left(
\|f\|_{L^\infty(\Omega)}^{p-2}
+
\|g\|_{L^\infty(\Omega)}^{p-2}
\right)
\left(
\|f\|_{\dot H^\sigma(\Omega)}
+
\|g\|_{\dot H^\sigma(\Omega)}
\right)
\|f-g\|_{L^\infty(\Omega)}.
\label{eq:second-product-term}
\end{align}
This is exactly the second term required in
\eqref{eq:Lip-chain-first}. \medskip

\noindent\textbf{$\bullet$ Step 2:} Combining \eqref{eq:product-B-h},
\eqref{eq:first-product-term} and
\eqref{eq:second-product-term}, we obtain the following uniform estimate for \(\theta\in[0,1]\):
\begin{align*}
\big\|D\mathcal{N}[h_\theta] h\big\|_{\dot H^\sigma(\Omega)}
&\le
C
\left(
\|f\|_{L^\infty(\Omega)}^{p-1}
+
\|g\|_{L^\infty(\Omega)}^{p-1}
\right)
\|f-g\|_{\dot H^\sigma(\Omega)}
\notag\\
&\quad+
C
\left(
\|f\|_{L^\infty(\Omega)}^{p-2}
+
\|g\|_{L^\infty(\Omega)}^{p-2}
\right) \Big(
\|f\|_{\dot H^\sigma(\Omega)}
+
\|g\|_{\dot H^\sigma(\Omega)}
\Big)
\|f-g\|_{L^\infty(\Omega)}.
\end{align*}
Applying Minkowski's integral inequality we arrive at
\begin{align}
\big\|\mathcal N[f]-\mathcal N[g]\big\|_{\dot H^\sigma(\Omega)}
&=
\left\|
\int_0^1 D\mathcal{N}[h_\theta] h\,d\theta
\right\|_{\dot H^\sigma(\Omega)}
\le
\int_0^1
\big\|D\mathcal{N}[h_\theta] h\big\|_{\dot H^\sigma(\Omega)}
\,d\theta \notag \\
&\le
C
\left(
\|f\|_{L^\infty(\Omega)}^{p-1}
+
\|g\|_{L^\infty(\Omega)}^{p-1}
\right)
\|f-g\|_{\dot H^\sigma(\Omega)} \notag \\
&\qquad+
C
\left(
\|f\|_{L^\infty(\Omega)}^{p-2}
+
\|g\|_{L^\infty(\Omega)}^{p-2}
\right) \left(
\|f\|_{\dot H^\sigma(\Omega)}
+
\|g\|_{\dot H^\sigma(\Omega)}
\right)
\|f-g\|_{L^\infty(\Omega)}.
\label{eq:whole-Sobolev-difference}
\end{align}
Finally, let us turn to the Dirichlet Laplacian to get the equivalent relation
\begin{equation*}
\|f\|_{\dot H_D^\sigma(\Omega)}
\sim
\|\mathcal A^{\sigma/2}f\|_{L^2(\Omega)}
\end{equation*}
for \(1<\sigma<2\). Consequently, one obtains the following estimates:
\begin{align*}
    \|f-g\|_{\dot H_D^\sigma(\Omega)}
&\lesssim
\big\|\mathcal A^{\sigma/2}(f-g)\big\|_{L^2(\Omega)}, \\
\|f\|_{\dot H_D^\sigma(\Omega)}
+
\|g\|_{\dot H_D^\sigma(\Omega)}
&\lesssim
\big\|\mathcal A^{\sigma/2}f\big\|_{L^2(\Omega)}
+
\big\|\mathcal A^{\sigma/2}g\big\|_{L^2(\Omega)}.
\end{align*}
In this way, plugging the two previous estimates into \eqref{eq:whole-Sobolev-difference} we may conclude the estimate \eqref{eq:Lip-chain-first}. Hence, our proof is established.
\end{proof}

\subsection{Linear estimates}
In this subsection, we collect some decay estimates for solutions to the corresponding linear damped wave equation. We denote by \(D(t,\mathcal{A})\), the solution operator to
\begin{equation}\label{eq:linear-D}
\begin{cases}
U_{tt}+U_t+ \mathcal{A}U=0, &(t,x)\in(0,\infty)\times\Omega,\\
U(t,x)=0, &(t,x)\in(0,\infty)\times \partial\Omega, \\
U(0,x)=0,\quad U_t(0,x)=f(x), & x\in \Omega,
\end{cases}
\end{equation}
to get the representation formula of solutions as follows: $U(t)=D(t,\mathcal{A})f$. At first, let us recall some decay estimates and the large time behavior of solutions to \eqref{eq:linear-D}, which plays key ingredients in our proof.

\begin{lemma}[Matsumura-type estimates, see \cite{IkedaSobajimaTaniguchiWakasugi2024}]\label{lem:Mats}
Let \(k=0,1\), \(q\in[1,2]\), \(s\ge0\), and assume that
\[
f\in L^q(\Omega)\cap H^{[k+s-1]^+}(\mathcal{A}).
\]
Then, there exists \(C>0\) such that the following estimates hold:
\begin{equation}\label{eq:Mats}
\big\|\partial_t^k \mathcal{A}^{s/2}D(t,\mathcal{A})f\big\|_{L^2(\Omega)}
\le
C(1+t)^{-\left(\frac1q-\frac12\right)-k-\frac{s}{2}}
\left(
\|f\|_{L^q(\Omega)}
+
e^{-\frac{t}{4}}\|f\|_{H^{[k+s-1]^+}(\mathcal{A})}
\right)
\end{equation}
for all \(t>0\), and
\begin{equation}
\label{eq:linear-diffusion-estimate-used}
\left\|
\partial_t^k A^{s/2}
\Big(D(t,\mathcal{A})-e^{-t\mathcal{A}}\Big)f
\right\|_{L^2(\Omega)}
\le
C(1+t)^{-\left(\frac1q-\frac12\right)-k-\frac{s}{2}-1}
\left(
\|f\|_{L^q(\Omega)}
+
e^{-\frac{t}{4}}\|f\|_{H^{k+s-1}(\mathcal{A})}
\right)
\end{equation}
for \(t\ge1\).
\end{lemma}
Next, let us write the solutions to
$$
\begin{cases}
u_{tt}+u_t+ \mathcal{A}u= 0, &(t,x)\in(0,\infty)\times\Omega,\\
u(t,x)=0, &(t,x)\in(0,\infty)\times \partial\Omega, \\
u(0,x)=u_0(x),\quad u_t(0,x)=u_1(x), & x\in \Omega.
\end{cases}
$$
by
\[
u_{\rm lin}(t)
=
\partial_tD(t,\mathcal{A})u_0+D(t,\mathcal{A})(u_0+u_1)
\]
to derive
\[
v_{\rm lin}(t):=\partial_tu_{\rm lin}(t)
=
-\mathcal{A}D(t,\mathcal{A})u_0+\partial_tD(t,\mathcal{A})u_1.
\]

\begin{proposition}
\label{lem:vlin}
Let \(q\in[1,2]\) and \(1<\sigma<2\). Assume that the initial data satisfy
\[
u_0\in L^q(\Omega)\cap H^{\sigma+1}(\mathcal{A}),
\qquad
u_1\in L^q(\Omega)\cap H^\sigma(\mathcal{A}).
\]
Then, there exists \(C>0\) such that for all \(t>0\) the following estimates hold:
\begin{equation}\label{eq:vlin-L2}
\|v_{\rm lin}(t)\|_{L^2(\Omega)}
\le
C(1+t)^{-\left(\frac1q+\frac12\right)}
\left(
\|u_0\|_{L^q(\Omega)}
+\|u_0\|_{H^1(\mathcal{A})}
+\|u_1\|_{L^q(\Omega)}
+\|u_1\|_{L^2(\Omega)}
\right)
\end{equation}
and
\begin{equation}\label{eq:vlin-high}
\big\|\mathcal{A}^{\sigma/2}v_{\rm lin}(t)\big\|_{L^2(\Omega)}
\le
C(1+t)^{-\left(\frac1q+\frac12+\frac{\sigma}{2}\right)}
\left(
\|u_0\|_{L^q(\Omega)}
+\|u_0\|_{H^{\sigma+1}(\mathcal{A})}
+\|u_1\|_{L^q(\Omega)}
+\|u_1\|_{H^\sigma(\mathcal{A})}
\right).
\end{equation} 
\end{proposition} 

\begin{proof}
First, applying \eqref{eq:Mats} in Lemma~\ref{lem:Mats} with $(k,s)=(1,0)$ and $(k,s)=(0,2)$ we get
\[
\big\|\partial_tD(t,\mathcal{A})u_1\big\|_{L^2(\Omega)}
\le
C(1+t)^{-\left(\frac1q+\frac12\right)}
\left(
\|u_1\|_{L^q(\Omega)}
+
e^{-t/4}\|u_1\|_{L^2(\Omega)}
\right)
\]
and
\[
\big\|\mathcal{A}D(t,\mathcal{A})u_0\big\|_{L^2(\Omega)}
\le
C(1+t)^{-\left(\frac1q+\frac12\right)}
\left(
\|u_0\|_{L^q(\Omega)}
+
e^{-t/4}\|u_0\|_{H^1(\mathcal{A})}
\right),
\]
respectively. Thus, combining these two inequalities gives the estimate \eqref{eq:vlin-L2}. For the higher-order estimate, we first apply \eqref{eq:Mats} in Lemma~\ref{lem:Mats} with $(k,s)=(1,\sigma)$ to deduce
\[
\big\|\mathcal{A}^{\sigma/2}\partial_tD(t,\mathcal{A})u_1\big\|_{L^2(\Omega)}
\le
C(1+t)^{-\left(\frac1q+\frac12+\frac{\sigma}{2}\right)}
\left(
\|u_1\|_{L^q(\Omega)}
+
e^{-t/4}\|u_1\|_{H^\sigma(\mathcal{A})}
\right).
\]
Again, using \eqref{eq:Mats} in Lemma~\ref{lem:Mats} with $(k,s)=(0,\sigma+2)$ we obtain
\[
\big\|\mathcal{A}^{(\sigma+2)/2}D(t,\mathcal{A})u_0\big\|_{L^2}
\le
C(1+t)^{-\left(\frac1q+\frac12+\frac{\sigma}{2}\right)}
\left(
\|u_0\|_{L^q(\Omega)}
+
e^{-t/4}\|u_0\|_{H^{\sigma+1}(\mathcal{A})}
\right).
\]
In this way, we combine the last two estimates to get the estimate \eqref{eq:vlin-high}. Hence, our proof is complete.
\end{proof}

\section{Proof of Theorem \ref{Main.Theorem}}\label{Sec.Proof}

\subsection{Setting of the change of variables}\label{Sec.Initial}

In this subsection, we reformulate the original problem (\ref{eq:main}) into an integral equation which should be suitable for applying the fixed-point argument. Since the nonlinearity of (\ref{eq:main}) depends only on the time derivative, it is natural to introduce the change of variables $v:=u_t$. The idea is based on constructing \(v\) first and then we will recover \(u\) by taking integration. Namely, thanks to Duhamel's principle, we express the solutions to to (\ref{eq:main}) by
\begin{equation}\label{eq:mild-u}
u(t)
=
\partial_tD(t,\mathcal{A})u_0
+
D(t,\mathcal{A})(u_0+u_1)
+
\int_0^t
D(t-s,\mathcal{A})\mathcal{N}[u_t](s)\,ds.
\end{equation}
Due to the fact $\partial_t^2 D+\partial_t D+\mathcal{A}D=0$, it entails
$\partial_t^2D=-\partial_tD-\mathcal{A}D$. Differentiating both sides of \eqref{eq:mild-u} with respect to $t$ we achieve
\begin{align*}
u_t(t)
&=
\partial_t^2D(t,\mathcal{A})u_0
+
\partial_tD(t,\mathcal{A})(u_0+u_1)
+
\int_0^t
\partial_tD(t-s,\mathcal{A})
\mathcal{N}[u_t](s)ds
\\
&=
-\mathcal{A}D(t,\mathcal{A})u_0
+
\partial_tD(t,\mathcal{A})u_1
+
\int_0^t
\partial_tD(t-s,\mathcal{A})
\mathcal{N}[u_t](s)ds,
\end{align*}
that is,
\begin{equation}\label{eq:mild-v}
v(t)
=
-\mathcal{A}D(t,\mathcal{A})u_0
+
\partial_tD(t,\mathcal{A})u_1
+
\int_0^t
\partial_tD(t-s,\mathcal{A})
\mathcal{N}[v](s)\,ds.
\end{equation}
In this way, we may recover the corresponding solution \(u\) to \eqref{eq:main} by the relation
\begin{equation}\label{eq:recover}
u(t)
:=
u_0
+
\int_0^t
v(\tau)\,d\tau.
\end{equation}
For this reason, let gives the following definition of mild solution to \eqref{eq:main}.

\begin{definition}[Mild solution]
Let \(T>0\). Then, a function
$$u\in \mathcal{C}^1\Big([0,T);D(\mathcal{A}^{\sigma/2})\Big)$$
is called a mild local solution to \eqref{eq:main} if $u$ can be written by (\ref{eq:recover}) and
$$ v\in \mathcal{C}\Big([0,T);D(\mathcal{A}^{\sigma/2})\Big)$$
satisfies \eqref{eq:mild-v}. If we can take $T$ as $\infty$, then $u$ is called a mild global solution to (\ref{eq:main}).
\end{definition}

\begin{remark}[Equivalence of formulations]
\label{lem:equivalence}
Let \(T>0\). If
$$ v\in \mathcal{C}\Big([0,T);D(\mathcal{A}^{\sigma/2})\Big)$$
satisfies
\eqref{eq:mild-v},
then the function
$$u\in \mathcal{C}^1\Big([0,T);D(\mathcal{A}^{\sigma/2})\Big)$$
defined by
\eqref{eq:recover}
satisfies
\eqref{eq:mild-u}. Conversely, if \(u\) satisfies
\eqref{eq:mild-u}, then $v$ satisfies \eqref{eq:mild-v}. From these observation, it shows that solving \eqref{eq:mild-v} is equivalent to solving \eqref{eq:main}.
\end{remark}
To give the proof of Theorem \ref{Main.Theorem}, the following auxiliary estimate is also useful.
\begin{lemma} \label{lem:conv}
Let \(a>0\), \(b>1\), and \(0<\rho<\min\{a,b\}\). Then, it holds
\[
\int_0^t
\langle t-s\rangle^{-a}
\langle s\rangle^{-b}\,ds
\le
C(1+t)^{-\rho}.
\]
\end{lemma}

\subsection{Proof of the global existence part}\label{Proof.Sec}
The proof is based on a contraction argument in a time-weighted function space adapted to the damped evolution.
A key point is to establish nonlinear estimates which is compatible with fractional powers of the Dirichlet Laplacian and the decay structure of the associated semigroup. At first, let us denote the following quantities:
\[
\gamma:=\frac1{\gamma(p)}+\frac12 \quad\text{and}\quad
\beta:=
\begin{cases}
p\gamma &\text{ if } 1<p<2,\\[1mm]
\dfrac{3p}{2} &\text{ if } p\ge2.
\end{cases}
\]
Due to the fact that \(\beta>\gamma\) for every \(p>1\), we choose \(\eta\) such that
\[
\gamma<\eta<\min\left\{\gamma+\frac{\sigma}{2},\beta\right\}
\]
with $1<\sigma<\min\{2,p\}$. For $T>0$, we introduce the evolution space
\[
X(T):= \mathcal{C}\Big([0,T);D(\mathcal{A}^{\sigma/2})\Big),
\]
endowed with the norm
\[
\|v\|_{X(T)}
:=
\sup_{0\le t<T}
\Big(
(1+t)^\gamma\|v(t)\|_{L^2(\Omega)}
+
(1+t)^\eta\|\mathcal{A}^{\sigma/2}v(t)\|_{L^2(\Omega)}
\Big).
\]
We define the following operator $\Phi$ on the space $X(T)$:
\begin{align*}
   \Phi[v](t,x) &= -\mathcal{A}D(t,\mathcal{A})u_0 + \partial_tD(t,\mathcal{A})u_1 + \int_0^t \partial_tD(t- s,\mathcal{A}) \mathcal{N}[v](s)\,ds \\
    &= : v_{\rm lin}(t)+v_{\rm nlin}(t).
\end{align*}
Our main purpose is to verify a pair of the following inequalities for all $v, w \in X(T)$:
    \begin{align}
        \big\|\Phi[v]\big\|_{X(T)} &\leq C_1 \|(u_0, u_1)\|_{\rm Data} + C_2\|v\|_{X(T)}^p,  \label{eq:self-final} \\
        \big\|\Phi[v]-\Phi[w]\big\|_{X(T)} &\leq C_3 \|v-w\|_{X(T)}\left(\|v\|_{X(T)}^{p-1}+\|w\|_{X(T)}^{p-1}\right) \label{eq:contract-final}.
    \end{align}
Then $\Phi$ is a contraction mapping to obtain a unique solution $v^* = \Phi[v^*] \in X(T)$ by the Banach fixed point theorem. Since $T$ is arbitrary, in this way we conclude that $v^* \in X(\infty)$. Namely, the proof can be divided into some steps as follows: \medskip

\noindent\textbf{\underline{Step 1:}} Let us show the inequality \eqref{eq:self-final} at first. By the linear estimates from Proposition \ref{lem:vlin}, we arrive at
\[
\|v_{\rm lin}\|_{X(T)}
\le
C_1 \|(u_0, u_1)\|_{\rm Data}
\]
For this reason, it is sufficient to prove the following estimate in place of \eqref{eq:self-final}:
\begin{equation} \label{eq:self-final-0}
    \|v_{\rm nlin}\|_{X(T)} \leq C_2\|v\|_{X(T)}^p.
\end{equation}
Indeed, setting $M:=\|v\|_{X(T)}$ we derive
\begin{align}
\|v(t)\|_{L^2(\Omega)}
&\le M(1+t)^{-\gamma},
\label{eq:self-v-L2}\\
\big\|A^{\sigma/2}v(t)\big\|_{L^2(\Omega)}
&\le M(1+t)^{-\eta}
\label{eq:self-v-high}
\end{align}
by the definition of \(X(T)\). Using the interpolation argument and \eqref{eq:self-v-L2}-\eqref{eq:self-v-high}, one finds
\begin{equation}\label{eq:self-v-inf}
\|v(t)\|_{L^\infty(\Omega)}
\le
C
\|v(t)\|_{L^2(\Omega)}^{1-\frac1\sigma}
\big\|\mathcal{A}^{\sigma/2}v(t)\big\|_{L^2(\Omega)}^{\frac1\sigma} \le
CM(1+t)^{-\delta},
\end{equation}
where
$$
\delta
:=
\gamma\left(1-\frac1\sigma\right)+\frac{\eta}{\sigma}.
$$
Since \(\eta>\gamma\), it implies $\delta>\gamma$. To estimate \(\mathcal{N}[v]\) in \(L^{\gamma(p)}(\Omega)\), let us divide our consideration into two cases as follows:
\begin{itemize}
    \item If \(1<p<2\), then \(\gamma(p)=2/p\) and one has
\[
\big\|\mathcal{N}[v](t)\big\|_{L^{\gamma(p)}(\Omega)}
\le C\big\||v(t)|^p\big\|_{L^{2/p}(\Omega)}
=
C\|v(t)\|_{L^2(\Omega)}^p
\le
M^p(1+t)^{-p\gamma}= M^p(1+t)^{-\beta}.
\]
    \item If \(p\ge2\), then \(\gamma(p)=1\), \(\gamma=3/2\) and we get
\begin{align*}
\big\|\mathcal{N}[v](t)\big\|_{L^1(\Omega)} &\le \|v(t)\|_{L^\infty(\Omega)}^{p-2} \|v(t)\|_{L^2(\Omega)}^2 \le CM^p (1+t)^{-(p-2)\delta- 2\gamma}.
\end{align*}
Since $\delta >\gamma$, it hold $(p-2)\delta +2\gamma
> 3p/2= \beta$ which means that 
\[
\big\|\mathcal{N}[v](t)\big\|_{L^1(\Omega)} \le CM^p(1+t)^{-\beta}.
\]
\end{itemize}
Therefore, both these cases lead to
\begin{equation}\label{eq:self-F-q}
\big\|\mathcal{N}[v](t)\big\|_{L^{\gamma(p)}(\Omega)}
\le
CM^p(1+t)^{-\beta}.
\end{equation} 
Next, applying the estimate \eqref{eq:chain} of the fractional chain rule from Proposition \ref{prop:chain} we achieve
\begin{equation}\label{eq:self-F-high}
\big\|\mathcal{A}^{\sigma/2}\mathcal{N}[v](t)\big\|_{L^2(\Omega)}
\le
C
\|v(t)\|_{L^\infty(\Omega)}^{p-1}
\big\|\mathcal{A}^{\sigma/2}v(t)\big\|_{L^2(\Omega)}
\le
CM^p
(1+t)^{-(p-1)\delta-\eta}.
\end{equation}
by \eqref{eq:self-v-high} and \eqref{eq:self-v-inf}. Now, the application of \eqref{eq:Mats} in Lemma~\ref{lem:Mats} with $k=1$, $s=0$, $q=\gamma(p)$ gives
\begin{align*}
\|v_{\rm nlin}(t)\|_{L^2(\Omega)}
&\le
C
\int_0^t
(1+t-s)^{-\gamma}
\big\|\mathcal{N}[v](s)\big\|_{L^{\gamma(p)}(\Omega)}\,ds + C\int_0^t e^{-\frac{t-s}{4}} \big\|\mathcal{N}[v](s)\big\|_{L^2(\Omega)}\,ds \nonumber \\
&\quad =: I_1(t) + I_2(t).
\end{align*}
By \eqref{eq:self-F-q}, the first integral is controlled as follows:
\[
I_1(t)
\le
CM^p
\int_0^t
(1+t-s)^{-\gamma}
(1+s)^{-\beta}\,ds
\le
CM^p(1+t)^{-\gamma},
\]
where we have utilized Lemma \ref{lem:conv} with $\gamma>0$, $\beta>1$ and $\gamma<\beta$. For the second integral, we use
\[
\big\|\mathcal{N}[v](t)\big\|_{L^2(\Omega)}
\le
\|v(t)\|_{L^\infty}^{p-1}\|v(t)\|_{L^2(\Omega)}
\le
CM^p
(1+t)^{-(p-1)\delta-\gamma}.
\]
and the integrability of exponential kernel to conclude
\[
I_2(t)
\le
CM^p(1+t)^{-\gamma},
\]
by noticing that $(p-1)\delta+\gamma>\gamma$. Therefore, combining the two previous estimates one arrives at
\begin{equation}\label{eq:self-N-L2}
\|v_{\rm nlin}(t)\|_{L^2(\Omega)}
\le
CM^p(1+t)^{-\gamma}.
\end{equation}
Again, the application of \eqref{eq:Mats} in Lemma~\ref{lem:Mats} with $k=1$, $s=\sigma$ and $q=\gamma(p)$ leads to
\begin{align*}
\big\|\mathcal{A}^{\sigma/2}v_{\rm nlin}(t)\big\|_{L^2(\Omega)}
&\le
C
\int_0^t
(1+t-s)^{-\gamma-\frac{\sigma}{2}}
\big\|\mathcal{N}[v](s)\big\|_{L^{\gamma(p)}(\Omega)}\,ds
+
C
\int_0^t
e^{-\frac{t-s}{4}}
\big\|\mathcal{N}[v](s)\big\|_{H^\sigma(\mathcal{A})}\,ds \notag \\
&\quad =: I_3(t)+ I_4(t).
\end{align*}
For the first integral, we use \eqref{eq:self-F-q} to proceed with
\[
I_3(t)
\le
CM^p
\int_0^t
(1+t-s)^{-\gamma-\frac{\sigma}{2}}
(1+s)^{-\beta}\,ds
\le
CM^p(1+t)^{-\eta}
\]
by applying Lemma \ref{lem:conv} with the condition
\[
\eta<\min\left\{\gamma+\frac{\sigma}{2},\beta\right\},
\]
By \eqref{eq:self-F-high}, we handle the second integral in the following way:
\[
I_4(t)
\le
CM^p
\int_0^t
e^{-\frac{t-s}{4}}
(1+s)^{-(p-1)\delta-\eta}\,ds
\le
CM^p(1+t)^{-\eta},
\]
where we note that $(p-1)\delta+\eta>\eta$. In this way, we may conclude
\begin{equation}\label{eq:self-N-high}
\big\|\mathcal{A}^{\sigma/2}v_{\rm nlin}(t)\big\|_{L^2(\Omega)}
\le
CM^p(1+t)^{-\eta}.
\end{equation}
The combination of \eqref{eq:self-N-L2} and \eqref{eq:self-N-high} gives \eqref{eq:self-final-0} from the definition of the norm in $X(T)$. \medskip

\noindent\textbf{\underline{Step 2:}} Let us turn to estimate \eqref{eq:contract-final}. Denoting
\[
M_v:=\|v\|_{X(T)},
\qquad
M_w:=\|w\|_{X(T)}
\quad \text{ and }\quad
M:=M_v+M_w.
\]
From \eqref{eq:N-Lip}, it follows the pointwise elementary inequality 
\begin{equation}\label{eq:pointwise-diff}
\big|\mathcal{N}[v]-\mathcal{N}[w]\big|
\le
C
\left(
|v|^{p-1}+|w|^{p-1}
\right)
|v-w|
\end{equation}
to get
\[
\|v(t)-w(t)\|_{L^2(\Omega)}
\le
(1+t)^{-\gamma}\|v-w\|_{X(T)},
\]
and
\[
\big\|\mathcal{A}^{\sigma/2}\big(v(t)-w(t)\big)\big\|_{L^2(\Omega)}
\le
(1+t)^{-\eta}\|v-w\|_{X(T)}.
\]
Also, an interpolation argument leads to
\[
\|v(t)\|_{L^\infty(\Omega)}
+
\|w(t)\|_{L^\infty(\Omega)}
\le
CM
(1+t)^{-\delta}.
\]
To control the norm of $\mathcal{N}[v]-\mathcal{N}[w]$ in \(L^{\gamma(p)}(\Omega)\), let us consider the following two cases:
\begin{itemize}
    \item If \(1<p<2\), then \(\gamma(p)=2/p\). Using \eqref{eq:pointwise-diff} and H\"{o}lder inequality one has
\[
\big\|\mathcal{N}[v](t)-\mathcal{N}[w](t)\big\|_{L^{\gamma(p)}(\Omega)}
\le
C
\left(
\|v\|_{L^2(\Omega)}^{p-1}
+
\|w\|_{L^2(\Omega)}^{p-1}
\right)
\|v-w\|_{L^2(\Omega)}.
\]
Therefore, it holds
\[
\big\|\mathcal{N}[v](t)-\mathcal{N}[w](t)\big\|_{L^{\gamma(p)}(\Omega)}
\le
CM^{p-1} (1+t)^{-p\gamma}\|v-w\|_{X(T)} = CM^{p-1}(1+t)^{-\beta}\|v-w\|_{X(T)}.
\]
    \item If \(p\ge2\), then \(\gamma(p)=1\). By \eqref{eq:pointwise-diff}, we achieve
\begin{align*}
\big\|\mathcal{N}[v](t)-\mathcal{N}[w](t)\big\|_{L^1(\Omega)}
&\le
C\left(\|v\|_{L^\infty(\Omega)}^{p-2}+\|w\|_{L^\infty(\Omega)}^{p-2}\right) \left(\|v\|_{L^2(\Omega)}+\|w\|_{L^2(\Omega)}\right)\|v-w\|_{L^2(\Omega)} \\
&\le
CM^{p-1}(1+t)^{-\frac{3p}{2}}\|v-w\|_{X(T)}= CM^{p-1}(1+t)^{-\beta}\|v-w\|_{X(T)}.
\end{align*}
\end{itemize}
Thus, both these cases gives immediately
\begin{equation}\label{eq:diff-F-q}
\big\|\mathcal{N}[v](t)-\mathcal{N}[w](t)\big\|_{L^{\gamma(p)}(\Omega)}
\le CM^{p-1}(1+t)^{-\beta}\|v-w\|_{X(T)}.
\end{equation}
Next, the application of the estimate \eqref{eq:Lip-chain-first} of the fractional chain rule from Proposition \ref{prop:chain} leads to 
\begin{align*}
&\big\|\mathcal{A}^{\sigma/2}\big(\mathcal{N}[v](t)-\mathcal{N}[w](t)\big)\big\|_{L^2(\Omega)} \notag \\
&\quad\le
C
\left(
\|v(t)\|_{L^\infty(\Omega)}^{p-1}
+
\|w(t)\|_{L^\infty(\Omega)}^{p-1}
\right)\big\|\mathcal{A}^{\sigma/2}\big(v(t)-w(t)\big)\big\|_{L^2(\Omega)}
\notag\\
&\qquad+
C
\left(
\|v(t)\|_{L^\infty(\Omega)}^{p-2}
+
\|w(t)\|_{L^\infty(\Omega)}^{p-2}
\right)\big\|\mathcal{A}^{\sigma/2}w(t)\big\|_{L^2(\Omega)}
\|v(t)-w(t)\|_{L^\infty(\Omega)}.
\end{align*}
Using the same interpolation estimate for \(v-w\), we obtain
\begin{equation}\label{eq:diff-F-high}
\big\|\mathcal{A}^{\sigma/2}\big(\mathcal{N}[v](t)-\mathcal{N}[w](t)\big)\big\|_{L^2(\Omega)}
\le
CM^{p-1}(1+t)^{-(p-1)\delta-\eta}\|v-w\|_{X(T)}.
\end{equation}
It is clear to see that
\[
\Phi[v](t)-\Phi[w](t)
=
\int_0^t
\partial_tD(t-s,\mathcal{A})
\big(
\mathcal{N}[v](s)-\mathcal{N}[w](s)
\big)
\,ds.
\]
The \(L^2\)-norm part is estimated by Lemma~\ref{lem:Mats} and the same way as in \eqref{eq:diff-F-q}. Hence, we have
\[
\big\|\Phi[v](t)-\Phi[w](t)\big\|_{L^2(\Omega)}
\le
CM^{p-1}
(1+t)^{-\gamma}\|v-w\|_{X(T)}.
\]
Similarly, the high-order part is controlled by Lemma~\ref{lem:Mats}, \eqref{eq:diff-F-q} and \eqref{eq:diff-F-high}. Therefore, one finds
\[
\big\|\mathcal{A}^{\sigma/2}\big(\Phi[v](t)-\Phi[w](t)\big)\big\|_{L^2(\Omega)}
\le
CM^{p-1}(1+t)^{-\eta}
\|v-w\|_{X(T)}.
\]
Thus, combining both the parts we obtain
\[
\big\|\Phi[v]-\Phi[w]\big\|_{X(T)}
\le
CM^{p-1}
\|v-w\|_{X(T)}.
\]
Since
\[
M^{p-1}
\le
C
\left(
\|v\|_{X(T)}^{p-1}
+
\|w\|_{X(T)}^{p-1}
\right),
\]
we may conclude \eqref{eq:contract-final}. \medskip

\noindent\textbf{\underline{Step 3:}} To complete the proof of Theorem \ref{Main.Theorem}, we recall from \eqref{eq:mild-v} the following relation:
$$
v(t)
=
-\mathcal{A}D(t,\mathcal{A})u_0
+
\partial_tD(t,\mathcal{A})u_1
+
\int_0^t
\partial_tD(t-s,\mathcal{A})\mathcal{N}[v](s)\,ds.
$$
Since \(D(\mathcal{A}^{\sigma/2})\) is a Hilbert space, \(X(T)\) is a Banach space. We solve the above integral equation in the Banach space
\(X(\infty)\). Indeed, we get
\begin{align*}
    \big\|\Phi[v]\big\|_{X(\infty)} &\le
C\varepsilon
+
C\|v\|_{X(\infty)}^p, \\
\big\|\Phi[v]-\Phi[w]\big\|_{X(\infty)} &\le C \left(
\|v\|_{X(\infty)}^{p-1} + \|w\|_{X(\infty)}^{p-1} \right) \|v-w\|_{X(\infty)}.
\end{align*}
where
\[
\varepsilon
= \|(u_0, u_1)\|_{\rm Data}:= 
\|u_0\|_{L^{\gamma(p)}(\Omega)}
+
\|u_0\|_{H^{\sigma+1}(\mathcal{A})}
+
\|u_1\|_{L^{\gamma(p)}(\Omega)}
+
\|u_1\|_{H^\sigma(\mathcal{A})}.
\]
Let us now choose $M:=2C\varepsilon$ and consider
\[
\mathcal{B}_M
:=
\{
v\in X(\infty):
\|v\|_{X(\infty)}\le M
\}.
\]
Assume that
\(\varepsilon>0\)
is sufficiently small such that $CM^{p-1}\le 1/4$ and $CM^p\le M/2$. Then for every \(v\in \mathcal{B}_M\), we derive
\begin{align*}
\big\|\Phi[v]\big\|_{X(\infty)}
&
\le
C\varepsilon+CM^p \le
M/2+M/2= M,
\end{align*}
which implies $\Phi(\mathcal{B}_M)\subset \mathcal{B}_M$. Furthermore, if
\(v,w\in \mathcal{B}_M\),
then \eqref{eq:contract-final} gives

\begin{align*}
\big\|\Phi[v]-\Phi[w]\big\|_{X(\infty)}
&
\le
2CM^{p-1}
\|v-w\|_{X(\infty)} \le
\frac12
\|v-w\|_{X(\infty)}.
\end{align*}
Hence, one claims that \(\Phi\) is a contraction on
\(\mathcal{B}_M\). Since \(X(\infty)\) is a Banach space, we apply the Banach's fixed-point theorem to conclude a unique function $v^*\in X(\infty)$ such that $v^*=\Phi(v^*)$. Finally, let us now reconstruct the original solution \(u\) by the relation \eqref{eq:recover}. Since
\[
v\in X(\infty)
\subset
\mathcal{C}\Big([0,\infty);D(\mathcal{A})^{\sigma/2}\Big),
\]
the integral in \eqref{eq:recover} is well-defined and
\[
u\in
\mathcal{C}^1\Big([0,\infty);D(\mathcal{A}^{\sigma/2})\Big)
.
\]
After substituting \(u_t=v\) into \eqref{eq:mild-v} and then integrating in time, we recover \eqref{eq:mild-u} by the relation $\partial_t^2 D+\partial_t D+\mathcal{A}D=0$. Therefore \(u\) is a global mild solution of \eqref{eq:main}. Moreover, the uniqueness follows immediately from the uniqueness of the fixed point in \(X(\infty)\). This completes our proof.

\subsection{Proof of the asymptotic profile part} \label{sec:diffusion-profile}
In this section, we identify the leading-order asymptotic profile of the time derivative \(u_t\) of the global achieved solutions to \eqref{eq:main}. To do this, let us use again the same notations as in Sections \ref{Sec.Initial} and \ref{Proof.Sec}. Moreover, recalling the nonlinear estimates established in Section \ref{Proof.Sec} we arrive at the following auxiliary ones:
\begin{align}
\label{eq:F-Lq-decay}
\big\|\mathcal{N}[v](t)\big\|_{L^{\gamma(p)}(\Omega)}
&\le
C\varepsilon^p(1+t)^{-\beta},\\
\label{eq:F-L2-decay}
\big\|\mathcal{N}[v](t)\big\|_{L^2(\Omega)}
&\le
C\varepsilon^p(1+t)^{-(p-1)\delta-\gamma},\\
\label{eq:F-Hsigma-decay}
\big\|\mathcal{A}^{\sigma/2}\mathcal{N}[v](t)\big\|_{L^2(\Omega)}
&\le
C\varepsilon^p(1+t)^{-(p-1)\delta-\eta}.
\end{align}
Since $\beta>1$, $(p-1)\delta+\gamma>1$ and $(p-1)\delta+\eta>1$, it follows that
\[
\mathcal{N}[v]
\in
L^1\Big((0,\infty);L^{\gamma(p)}(\Omega)\Big)
\cap
L^1\Big((0,\infty);H^\sigma(\mathcal{A})\Big).
\]
Therefore, the Bochner integral
\begin{equation*}
\label{eq:nonlinear-asymptotic-data}
M_\infty
:=
\int_0^\infty \mathcal{N}[v](s)\,ds
\end{equation*}
is well defined in $L^{\gamma(p)}(\Omega)\cap H^\sigma(\mathcal{A})$. To get started, let us now decompose
\begin{equation}
\label{eq:ut-profile-decomposition}
u_t(t)+\mathcal{A}e^{-t\mathcal{A}}(u_0+u_1+M_\infty)
=
R_{\mathrm{lin}}(t)
+
R_{\mathrm{dw}}(t)
-
R_{\mathrm{shift}}(t)
+
R_{\mathrm{tail}}(t),
\end{equation}
where
\begin{align*}
R_{\mathrm{lin}}(t)
&:=
-\mathcal{A}\bigl(D(t,\mathcal{A})-e^{-t\mathcal{A}}\bigr)u_0
+
\partial_t\bigl(D(t,\mathcal{A})-e^{-t\mathcal{A}}\bigr)u_1,\\
R_{\mathrm{dw}}(t)
&:=
\int_0^t
\left(
\partial_tD(t-s,\mathcal{A})
+
\mathcal{A}e^{-(t-s)\mathcal{A}}
\right)\mathcal{N}[v](s)\,ds,\\
R_{\mathrm{shift}}(t)
&:=
\int_0^t
\mathcal{A}\left(
e^{-(t-s)\mathcal{A}}-e^{-t\mathcal{A}}
\right)\mathcal{N}[v](s)\,ds,\\
R_{\mathrm{tail}}(t)
&:=
\mathcal{A}e^{-t\mathcal{A}}\int_t^\infty \mathcal{N}[v](s)\,ds.
\end{align*}
We are going to divide our consideration into some steps as follows:\medskip

\noindent\textbf{$\bullet$ Estimate of the linear remainder $R_{\mathrm{lin}}(t)$:} After applying \eqref{eq:linear-diffusion-estimate-used} in Lemma~\ref{lem:Mats} with $(k,s)=(0,2)$ to \(u_0\), and with $(k,s)=(1,0)$
to \(u_1\), we obtain
\begin{align}
\label{eq:Rlin-little-o}
\|R_{\mathrm{lin}}(t)\|_{L^2(\Omega)}
&\le
C(1+t)^{-\frac{1}{\gamma(p)}-\frac{3}{2}}
\left(
\|u_0\|_{L^{\gamma(p)}(\Omega)}
+
\|u_0\|_{H^1(\mathcal{A})}
+
\|u_1\|_{L^{\gamma(p)}(\Omega)}
+
\|u_1\|_{L^2(\Omega)}
\right) = o(t^{-\gamma})
\end{align}
since $1/\gamma(p)+1/2=\gamma$. \medskip

\noindent\textbf{$\bullet$ Estimate of the damped-wave/heat difference $R_{\mathrm{dw}}(t)$:} We separate it into the following three sub-terms:
\[
R_{\mathrm{dw}}(t)
=
R_{\mathrm{dw}}^{(1)}(t)
+
R_{\mathrm{dw}}^{(2)}(t)
+
R_{\mathrm{dw}}^{(3)}(t),
\]
where
\begin{align*}
R_{\mathrm{dw}}^{(1)}(t)
&:=
\int_0^{t/2}
\left[
\partial_tD(t-s,\mathcal{A})+\mathcal{A}e^{-(t-s)\mathcal{A}}
\right]\mathcal{N}[v](s)\,ds,\\
R_{\mathrm{dw}}^{(2)}(t)
&:=
\int_{t/2}^{t-1}
\left[
\partial_tD(t-s,\mathcal{A})+\mathcal{A}e^{-(t-s)\mathcal{A}}
\right]\mathcal{N}[v](s)\,ds,\\
R_{\mathrm{dw}}^{(3)}(t)
&:=
\int_{t-1}^{t}
\left[
\partial_tD(t-s,\mathcal{A})+\mathcal{A}e^{-(t-s)\mathcal{A}}
\right]\mathcal{N}[v](s)\,ds,
\end{align*}
for \(t\ge4\). For the first part, since \(t-s\ge t/2\), the linear diffusion
estimate with \((k,s)=(1,0)\) gives
\begin{align}
\big\|R_{\mathrm{dw}}^{(1)}(t)\big\|_{L^2(\Omega)}
&\le
C\int_0^{t/2}
(1+t-s)^{-\frac{1}{\gamma(p)}-\frac{3}{2}}
\left(
\big\|\mathcal{N}[v](s)\big\|_{L^{\gamma(p)}(\Omega)}
+
e^{-\frac{t-s}{4}}\big\|\mathcal{N}[v](s)\big\|_{L^2(\Omega)}
\right)ds \notag\\
&\le
Ct^{-\frac{1}{\gamma(p)}-\frac{3}{2}}
\int_0^\infty\big\|\mathcal{N}[v](s)\big\|_{L^{\gamma(p)}(\Omega)}\,ds
+
Ce^{-ct}
\int_0^{t/2}\big\|\mathcal{N}[v](s)\big\|_{L^2(\Omega)}\,ds \notag\\
&\le
C\varepsilon^p t^{-\frac{1}{\gamma(p)}-\frac{3}{2}},
\label{eq:Rdw1-estimate}
\end{align}
where $c$ is suitable positive constant. For the second part, using the relation \(t-s\ge1\) one finds
\begin{align*}
\big\|R_{\mathrm{dw}}^{(2)}(t)\big\|_{L^2(\Omega)}
&\le
C\int_{t/2}^{t-1}
(t-s)^{-\frac{1}{\gamma(p)}-\frac{3}{2}}\big\|\mathcal{N}[v](s)\big\|_{L^{\gamma(p)}(\Omega)}\,ds \notag\\
&\quad
+
C\int_{t/2}^{t-1}
e^{-\frac{t-s}{4}}\big\|\mathcal{N}[v](s)\big\|_{L^2(\Omega)}\,ds.
\end{align*}
Due to \(s\ge t/2\), these estimates
\eqref{eq:F-Lq-decay} and \eqref{eq:F-L2-decay} leads to
\begin{align}
\big\|R_{\mathrm{dw}}^{(2)}(t)\big\|_{L^2(\Omega)}
&\le
C\varepsilon^p t^{-\beta}
\int_1^{t/2}s^{-\frac{1}{\gamma(p)}-\frac{3}{2}}\,ds
+
C\varepsilon^p t^{-(p-1)\delta-\gamma}
\int_1^\infty e^{-\frac{s}{4}}\,ds \notag\\
&\le
C\varepsilon^p
\left(
t^{-\beta}+t^{-(p-1)\delta-\gamma}
\right). \label{eq:Rdw2-estimate}
\end{align}
Concerning the last part we use the following short-time estimates for $0\le s\le1$:
\begin{align*}
    \big\|\partial_tD(s,\mathcal{A})f\big\|_{L^2(\Omega)} \le C\|f\|_{L^2(\Omega)} \quad \text{ and }\quad 
    \big\|\mathcal{A}e^{-s \mathcal{A}}f\big\|_{L^2(\Omega)} \le Cs^{-1+\frac{\sigma}{2}}\big\|\mathcal{A}^{\sigma/2}f\big\|_{L^2(\Omega)},
\end{align*}
since the spectral theorem combined with
$$
\mathcal{A}e^{-s \mathcal{A}}
=
\mathcal{A}^{1-\sigma/2}e^{-s \mathcal{A}}\mathcal{A}^{\sigma/2}
\quad \text{ and }\quad
\sup_{\lambda\ge0}
\lambda^{-1+\frac{\sigma}{2}}e^{-s\lambda}
\le
Cs^{-1+\frac{\sigma}{2}}.
$$
Consequently, one realizes
\begin{align}
\big\|R_{\mathrm{dw}}^{(3)}(t)\big\|_{L^2(\Omega)}
&\le
C\int_{t-1}^{t}\big\|\mathcal{N}[v](s)\big\|_{L^2(\Omega)}\,ds+
C\int_{t-1}^{t}
(t-s)^{-1+\frac{\sigma}{2}}
\big\|\mathcal{A}^{\sigma/2}\mathcal{N}[v](s)\big\|_{L^2(\Omega)}\,ds \notag\\
&\le
C\varepsilon^p t^{-(p-1)\delta-\gamma}
+
C\varepsilon^p t^{-(p-1)\delta-\eta}
\int_0^1s^{-1+\frac{\sigma}{2}},ds \notag\\
&\le
C\varepsilon^p
\left(
t^{-(p-1)\delta-\gamma}
+
t^{-(p-1)\delta-\eta}
\right)
\label{eq:Rdw3-estimate}
\end{align}
by \eqref{eq:F-L2-decay}. For this reason, we link all estimates \eqref{eq:Rdw1-estimate},
\eqref{eq:Rdw2-estimate} and
\eqref{eq:Rdw3-estimate} to obtain
\begin{equation}
\label{eq:Rdw-little-o}
\|R_{\mathrm{dw}}(t)\|_{L^2(\Omega)}
=
o(t^{-\gamma}).
\end{equation}

\noindent\textbf{$\bullet$ Estimate of the time-shift remainder $R_{\mathrm{shift}}$:} Let us now split the term
\[
R_{\mathrm{shift}}(t)
=
R_{\mathrm{shift}}^{(1)}(t)
+
R_{\mathrm{shift}}^{(2)}(t)
+
R_{\mathrm{shift}}^{(3)}(t)
\]
over the intervals $[0,t/2]$, $[t/2,t-1]$, $[t-1,t]$, respectively. For \(0\le s\le t/2\), the semigroup identity
\[
e^{-(t-s)\mathcal{A}}-e^{-t\mathcal{A}}
=
\int_{t-s}^{t}\mathcal{A}e^{-\tau \mathcal{A}}\,d\tau
\]
gives
\[
\mathcal{A}\left(e^{-(t-s)\mathcal{A}}-e^{-t\mathcal{A}}\right)
=
\int_{t-s}^{t}\mathcal{A}^2e^{-\tau \mathcal{A}}\,d\tau.
\]
Using the heat semigroup estimate
\begin{equation*}
\label{eq:heat-semigroup-A2}
\big\|\mathcal{A}^2e^{-\tau \mathcal{A}}f\big\|_{L^2(\Omega)}
\le
C\tau^{-\frac{1}{\gamma(p)}-\frac{3}{2}}\|f\|_{L^{\gamma(p)}(\Omega)},
\end{equation*}
we derive
\begin{align*}
\big\|R_{\mathrm{shift}}^{(1)}(t)\big\|_{L^2(\Omega)}
&\le
C\int_0^{t/2}
\int_{t-s}^{t}
\tau^{-\frac{1}{\gamma(p)}-\frac{3}{2}}\|\mathcal{N}[v](s)\|_{L^{\gamma(p)}(\Omega)}\,d\tau\,ds \notag\\
&\le
Ct^{-\frac{1}{\gamma(p)}-\frac{3}{2}}
\int_0^{t/2}
s\|\mathcal{N}[v](s)\|_{L^{\gamma(p)}(\Omega)}\,ds \notag\\
&\le
C\varepsilon^p t^{-\frac{1}{\gamma(p)}-\frac{3}{2}}
\int_0^{t/2}
s(1+s)^{-\beta}\,ds
\end{align*}
by \eqref{eq:F-Lq-decay}. Due to the fact
\[
\int_0^{t/2}
s(1+s)^{-\beta}\,ds
\le
C
\begin{cases}
1 &\text{ if }\beta>2,\\
\log(2+t) &\text{ if }\beta=2,\\
t^{2-\beta} &\text{ if }1<\beta<2,
\end{cases}
\]
it follows that
\begin{equation}
\label{eq:Rshift1-little-o}
\big\|R_{\mathrm{shift}}^{(1)}(t)\big\|_{L^2(\Omega)}
=
o\Big(t^{-\frac{1}{\gamma(p)}-\frac{1}{2}}\Big)
=
o(t^{-\gamma}).
\end{equation}
For the middle-time part, thanks to the triangle inequality and the heat semigroup estimates, one has
\begin{align*}
\big\|R_{\mathrm{shift}}^{(2)}(t)\big\|_{L^2(\Omega)}
&\le
\int_{t/2}^{t-1}
\big\|\mathcal{A}e^{-(t-s)\mathcal{A}}\mathcal{N}[v](s)\big\|_{L^2(\Omega)}\,ds +
\int_{t/2}^{t-1}
\big\|\mathcal{A}e^{-t\mathcal{A}}\mathcal{N}[v](s)\big\|_{L^2(\Omega)}\,ds \notag\\
&\le
C\int_{t/2}^{t-1}
(t-s)^{-\frac{1}{\gamma(p)}-\frac{1}{2}}\|\mathcal{N}[v](s)\|_{L^{\gamma(p)}(\Omega)}\,ds \notag\\
&\quad+
Ct^{-\frac{1}{\gamma(p)}-\frac{1}{2}}
\int_{t/2}^{t-1}\|\mathcal{N}[v](s)\|_{L^{\gamma(p)}(\Omega)}\,ds
\end{align*}
by \eqref{eq:F-Lq-decay}. Therefore, we may conclude that
\begin{align}
\big\|R_{\mathrm{shift}}^{(2)}(t)\big\|_{L^2(\Omega)}
&\le
C\varepsilon^p t^{-\beta}
\int_1^{t/2}s^{-\frac{1}{\gamma(p)}-\frac{1}{2}}\,ds
+
C\varepsilon^p t^{-\frac{1}{\gamma(p)}+\frac{1}{2}-\beta} \notag\\
&\le
C\varepsilon^p
\left(
t^{-\beta}+t^{-\frac{1}{\gamma(p)}+\frac{1}{2}-\beta}
\right)= o(t^{-\gamma})
\label{eq:Rshift2-estimate}
\end{align}
because of the fact \(\beta>\gamma= 1/\gamma(p)+1/2\). Finally, for the near-time part it yields from \eqref{eq:F-Lq-decay} and \eqref{eq:F-Hsigma-decay} that
\begin{align}
\big\|R_{\mathrm{shift}}^{(3)}(t)\big\|_{L^2(\Omega)}
&\le
C\int_{t-1}^{t}
(t-s)^{-1+\frac{\sigma}{2}}
\big\|\mathcal{A}^{\sigma/2}\mathcal{N}[v](s)\big\|_{L^2(\Omega)}\,ds +Ct^{-\frac{1}{\gamma(p)}-\frac{1}{2}} 
\int_{t-1}^{t}\big\|\mathcal{N}[v](s)\big\|_{L^{\gamma(p)}(\Omega)}\,ds \notag\\
&\le
C\varepsilon^p
\left(
t^{-(p-1)\delta-\eta}
+
t^{-\frac{1}{\gamma(p)}-\frac{1}{2}-\beta} 
\right).
\label{eq:Rshift3-estimate}
\end{align}
Since \((p-1)\delta+\eta>\gamma\), combining all estimates \eqref{eq:Rshift1-little-o}, \eqref{eq:Rshift2-estimate} and \eqref{eq:Rshift3-estimate} we arrive at
\begin{equation}
\label{eq:Rshift-little-o}
\big\|R_{\mathrm{shift}}(t)\big\|_{L^2(\Omega)}
=
o(t^{-\gamma}).
\end{equation}

\noindent\textbf{$\bullet$ Estimate of the nonlinear tail:} It is obvious to see that the \(L^{\gamma(p)}\)-\(L^2\) heat semigroup estimate gives
\[
\big\|\mathcal{A}e^{-t\mathcal{A}}f\big\|_{L^2(\Omega)}
\le
Ct^{-\frac{1}{\gamma(p)}-\frac{1}{2}} \|f\|_{L^{\gamma(p)}(\Omega)}
=
Ct^{-\gamma}\|f\|_{L^{\gamma(p)}(\Omega)}
\]
which immediately implies
\begin{align}
\big\|R_{\mathrm{tail}}(t)\big\|_{L^2(\Omega)}
&\le
Ct^{-\gamma}
\int_t^\infty\big\|\mathcal{N}[v](s)\big\|_{L^{\gamma(p)}(\Omega)}\,ds \notag\\
&\le
C\varepsilon^p t^{-\gamma}
\int_t^\infty (1+s)^{-\beta}\,ds \le
C\varepsilon^p t^{-\gamma-\beta+1}=
o(t^{-\gamma}),
\label{eq:Rtail-estimate}
\end{align}
where we have utilized \eqref{eq:F-Lq-decay} and the condition \(\beta>1\). \medskip

Summarizing, plugging the achieved estimates \eqref{eq:Rlin-little-o},
\eqref{eq:Rdw-little-o},
\eqref{eq:Rshift-little-o}, and
\eqref{eq:Rtail-estimate} in
\eqref{eq:ut-profile-decomposition} we have shown that
\[
\left\|
u_t(t)+\mathcal{A}e^{-t\mathcal{A}}(u_0+u_1+M_\infty)
\right\|_{L^2(\Omega)}
=
o(t^{-\gamma}),
\]
which is to verify the desired estimate \eqref{eq:ut-diffusion-limit}. In this way, our proof is complete.

\section{Further remarks} \label{Sec.Final}
\begin{remark}
    Throughout this paper, we have succeeded in not only proving the global (in time) existence of small data solutions but also analyzing the asymptotic profile of global solutions to \eqref{eq:main} in $2$D exterior domains to give a positive answer for an open problem in \cite{DuongDao}. It can be also expected to study such results for \eqref{eq:main} in higher dimensional exterior domains even taking account of the Neumann/Robin boundary conditions are of interest instead of the Dirichlet boundary. For this reason, we are going to develop new tools in terms of improving further our results for this problem in higher dimensions ($n\ge 3$) in the forthcoming work.
\end{remark}

\begin{remark}
    Another very interesting question, which should be proposed from this work, is to explore the corresponding weakly coupled system of \eqref{eq:main}. In other words, the following initial-boundary value problem for semi-linear damped wave equations with the power nonlinearity of derivative type in exterior domains is of interest:
    \begin{equation}\label{sys:main}
    \begin{cases}
        u_{tt}+u_t+ \mathcal{A}u= \mathcal{N}[v_t], & (t,x)\in(0,\infty)\times\Omega,\\
        v_{tt}+v_t+ \mathcal{A}v= \mathcal{N}[u_t], & (t,x)\in(0,\infty)\times\Omega,\\
        u(t,x)=v(t,x)=0, &(t,x)\in(0,\infty)\times \partial\Omega, \\
        u(0,x)=u_0(x),\quad u_t(0,x)=u_1(x), & x\in \Omega, \\
        v(0,x)=v_0(x),\quad v_t(0,x)=v_1(x), & x\in \Omega.
    \end{cases}
    \end{equation}
    At this point, our aim in the future work concerns with showing both the global (in time) existence of small data solutions to \eqref{sys:main} and a blow-up result at least in $2$D and then in higher dimensional cases. For this observation, one may recognize that the so-called critical curve for \eqref{sys:main} is well-established, which has never appeared in previous studies.
\end{remark}

\section*{Acknowledgements}
This work of Tuan Anh Dao is supported by Vietnam Ministry of Education and Training and Vietnam Institute for Advanced Study in Mathematics under grant number B2026-CTT-04. Masahiro Ikeda is supported by JSPS KAKENHI Grant Number JP25K24910 and JP23K03174.

\section*{Data availability}
Data sharing not applicable to this article as no datasets were generated or analyzed during the current study.

\section*{Conflict of interest}
The authors declare that they have no conflict of interest.

\end{document}